\documentclass[10pt,leqno]{amsart}
\usepackage{hyperref,amsmath,amsfonts,amssymb,amsthm}
\usepackage[numbers]{natbib}
\usepackage[mathcal]{euscript}
\usepackage{enumitem}
\setlist[enumerate]{label={\rm(\roman*)}}
\usepackage[a5paper, left=1cm, right=1cm, top=1.5cm, bottom =1cm]{geometry}

\newcommand{\R}{\mathbb{R}}
\newcommand{\N}{\mathbb{N}}

\newcommand{\Rn}{\mathbb{R}^n}
\renewcommand{\d}{{\fam0 d}}
\newcommand{\Liloc}{L^1_{\text{loc}}}
\newcommand{\eprec}{\prec\kern-3pt\prec}
\newcommand{\esucc}{\succ\kern-3pt\succ}
\newcommand{\gn}{{\frac{\gamma}{n}}}
\newcommand{\gnd}{{\gamma/n}}
\renewcommand{\ng}{{\frac{n}{\gamma}}}

\newcommand{\ngn}{{\frac{n}{n-\gamma}}}
\newcommand{\ngnd}{{n/(n-\gamma)}}
\newcommand{\nng}{{\frac{n}{\gamma-n}}}
\newcommand{\Ang}{{A_\ng}}
\newcommand{\Angi}{{A_\ng^{-1}}}
\newcommand{\Gng}{G_\ng}
\newcommand{\Gngi}{G_\ng^{-1}}
\newcommand{\Bng}{{B_\ng}}
\newcommand{\Bngi}{{B_\ng^{-1}}}
\newcommand{\Eng}{E_\ng}
\newcommand{\Engi}{E_\ng^{-1}}
\newcommand{\Fng}{F_\ng}
\newcommand{\Fngi}{F_\ng^{-1}}
\newcommand{\Asup}{{A_{\sup}}}
\newcommand{\Asupi}{{A_{\sup}^{-1}}}
\newcommand{\Gsup}{{G_{\sup}}}
\newcommand{\Gsupi}{{G_{\sup}^{-1}}}

\let\tilde\widetilde

\newtheoremstyle{MyPlain}
	{}{}{\itshape}{}
	{\bfseries}{.}
	{5pt plus 1pt minus 1pt}
	{\thmname{#1}\thmnumber{ #2}\thmnote{ \textbf{[#3]}}}
\theoremstyle{MyPlain}
\newtheorem{theorem}{Theorem}[section]
\newtheorem{corollary}[theorem]{Corollary}
\newtheorem{lemma}[theorem]{Lemma}
\newtheorem{proposition}[theorem]{Proposition}
\newtheoremstyle{MyRemark}
	{}{}{\upshape}{}
	{\bfseries}{.}
	{5pt plus 1pt minus 1pt}
	{\thmname{#1}\thmnumber{ #2}\thmnote{ \textbf{[#3]}}}
\theoremstyle{MyRemark}
\newtheorem{remark}[theorem]{Remark}
\newtheorem{example}[theorem]{Example}

\expandafter\let\expandafter\oldproof\csname\string\proof\endcsname
\let\oldendproof\endproof
\renewenvironment{proof}[1][\proofname]{%
  \oldproof[{{\bf #1}}]%
}{\oldendproof}

\numberwithin{equation}{section}

\begin{document}

\title{Fractional Maximal Operator in Orlicz spaces}

\begin{abstract}
We study the fractional maximal operator acting between Orlicz spaces.  We
characterise whether the operator is bounded between two given Orlicz spaces.
Also a necessary and sufficient conditions for the existence of an optimal
target and domain spaces are established and the explicit formulas of
corresponding optimal Young functions are presented.
\end{abstract}

\author{V\'\i t Musil}
\address{Department of Mathematical Analysis,
Faculty of Mathematics and Physics,
Charles University,
So\-ko\-lo\-vsk\'a~83,
186~75 Praha~8,
Czech Republic}
\email{musil@karlin.mff.cuni.cz}
\urladdr{0000-0001-6083-227X}

\subjclass[2000]{}

\thanks{%
This research was partly supported
by the grant P201-18-00580S of the Grant Agency of the Czech Republic,
by the Charles University, project GAUK No.~33315
and by the grant 8X17028 of the Czech Ministry of Education}

\date{\today}

\maketitle
\bibliographystyle{plain}

\section{Introduction}

Let $n\in\N$ and $0<\gamma<n$ be fixed.
The fractional maximal operator $M_\gamma$ is defined
for any locally integrable function $f$ in $\Rn$ by
\[
 M_\gamma f(x) = \sup_{Q\ni x} |Q|^{\ng-1} \int_{Q} |f(y)|\,\d y,
 	\quad x\in \Rn,
\]
where the supremum is taken over all cubes $Q$ containing $x$ and having the sides
parallel to the coordinate axes.

Our first aim is to analyse the boundedness of $M_\gamma$ acting between Orlicz spaces.
More specifically, given Orlicz spaces $L^A(\Rn)$ and $L^B(\Rn)$, we want to
decide whether
\begin{equation} \label{eq:MgAB}
	M_\gamma \colon L^A(\Rn) \to L^B(\Rn).
\end{equation}
We show that \eqref{eq:MgAB} is equivalent to one-dimensional inequalities
involving only the Young functions $A$ and $B$. Resulting inequality is then
much more easier to verify. Such a simplification is often called reduction principle.
Similar reductions in Orlicz spaces have already appeared in the literature for various types
of operators, let us mention for instance the reduction principle
for the Hardy-Littlewood maximal operator \citep{Kit:97}, fractional integrals and Riesz
potential \citep{Cia:99a}, Sobolev and Poincar\'e inequalities \citep{Cia:06},
Sobolev embeddings \citep{Cia:96,Cia:18}, Sobolev trace embeddings \citep{Cia:10}
or Korn type inequalities \citep{Cia:14a}.
The reduction principle for fractional maximal operator under
some restrictive assumptions can be found in \citep{Har:02}.
Our general result is stated in Theorem~\ref{thm:Orlicz_reduction}.

The principal innovation of this paper however lays
in the description of optimal Orlicz spaces
in \eqref{eq:MgAB}.
More
specifically, given $L^A(\Rn)$, we seek for the smallest Orlicz space
$L^B(\Rn)$ such that \eqref{eq:MgAB} holds. By ``smallest'' we mean that if
\eqref{eq:MgAB} holds with $L^B(\Rn)$ replaced with another Orlicz space
$L^{\widehat B}(\Rn)$, then $L^B(\Rn) \to L^{\widehat B}(\Rn)$. Instead of
smallest we also often say optimal.

Let us briefly look at the situation in the class of the Lebesgue spaces.
It is well known that
\[
	M_\gamma\colon L^p(\Rn) \to
	\begin{cases}
		L^{\frac{np}{n-\gamma p}}(\Rn),
			& 1<p<\tfrac n\gamma,
				\\
		L^\infty(\Rn),
			& p=\tfrac n\gamma,
	\end{cases}
\]
and this result is sharp within Lebesgue spaces.
However, there is no Lebesgue space $L^q(\Rn)$
for which
\[
	M_\gamma\colon L^1(\Rn) \to L^q(\Rn).
\]
The situation in the class of Orlicz spaces is much more subtle
and not many results are available in the literature.
The authors of \citep{Edm:02}
characterised the boundedness of $M_\gamma$ (they work with more general
operator, in fact) on classical Lorentz spaces. In the cases when such spaces
coincide with Orlicz spaces, we may recover the following result
(see Section~\ref{subsec:Orlicz} for the definitions of the spaces involved).
\begin{equation} \label{ex:EO}
	M_\gamma\colon L^p (\log L)^{\mathbb A}(\Rn) \to
		\begin{cases}
			L^{\ngn}(\log L)^{\frac{n\mathbb A}{n-\gamma} - 1}(\Rn),
				& \text{$p=1$, $\alpha_0<0$, $\alpha_\infty > 0$},
				\\
			L^{\frac{np}{n-\gamma p}} (\log L)^{\frac{n\mathbb A}{n-\gamma p}}(\Rn),
				& \text{$1<p<\tfrac{n}{\gamma}$},
				\\
			\exp L^{-\frac{n}{\gamma\mathbb A}}(\Rn),
				& \text{$p=\tfrac{n}{\gamma}$, $\alpha_0>0$, $\alpha_\infty < 0$},
		\end{cases}
\end{equation}
however, it does not say anything about its sharpness.
In Theorem~\ref{thm:OOT}, we give the complete characterization
of the existence of the optimal Orlicz target and, in the
affirmative case, we give its full description.
It turns out that the spaces obtained in \eqref{ex:EO}
are optimal in the cases when $p>1$. If $p=1$, then
the target space $L^{\ngn}(\log L)^{\frac{n\mathbb A}{n-\gamma} - 1}(\Rn)$
is not the best possible Orlicz space and even more,
the optimal Orlicz target does not exist in this case.
This means that any Orlicz space $L^B(\Rn)$ for which
$M_\gamma\colon L^p(\log L)^{\mathbb A}(\Rn)\to L^B(\Rn)$
may be replaced by essentially smaller space, whence
there is an ``open'' set of all the eligible Orlicz spaces.
The details on this particular case are discussed in
Example~\ref{ex:LZrange}.

One can also ask the converse problem, i.e.~when the target space
 $L^{B}(\Rn)$ is given and we seek for the
largest possible $L^A(\Rn)$ rendering \eqref{eq:MgAB} true.
Analogously, by ``largest'' we mean that if \eqref{eq:MgAB} holds
with $L^A(\Rn)$ replaced by $L^{\widehat A}(\Rn)$, then 
$L^{\widehat A}(\Rn) \to L^A(\Rn)$.
Again, we shorten this notion to the word ``optimal'' since
no confusion with the above situation is likely to happen.
The solution of this task is the subject of Theorem~\ref{thm:OOD},
where we give the complete description of optimal domains.
If we return back to the example in \eqref{ex:EO}, one
gets that the Orlicz domain $L^A(\log L)^{\mathbb A}(\Rn)$
is the optimal one in all three cases. See Example~\ref{ex:LZdomain}
for further details.

The paper is organized as follows.
The main results are stated in Section~\ref{sec:main} as well as examples of particular
Orlicz spaces. In the Section~\ref{sec:back} we collect necessary preliminary material and
in Section~\ref{sec:proofs} we prove the results.

\section{Background} \label{sec:back}

We call $A\colon [0, \infty) \to [0, \infty]$  a Young function if it is  convex, left-continuous, and $A(0)=0$. Any function of this kind satisfies, in particular,
\begin{equation}\label{kt}
kA(t) \leq A(kt)
\quad \text{if $k \geq 1$ and $t \geq 0$.}
\end{equation}
Every Young function admits the representation
\[
	A(t) = \int_{0}^{t} a(s)\,\d s
		\quad \text{for $t\ge 0$,}
\]
where $a\colon[0,\infty)\to [0,\infty]$ is nondecreasing
and left-continuous function. It holds that
\begin{equation} \label{eq:A}
	A(t) \le t a(t) \le A(2t)
		\quad \text{for $t\ge 0$.}
\end{equation}
The Young conjugate 
 $\widetilde{A}$ of  $A$ is given by
\begin{equation*}
    \widetilde{A}(t)=\sup \{st-A(s): s \geq 0\} \quad \text{for $t\ge 0$.}
\end{equation*}
The  function $\widetilde{A}$ is a Young function as well, and its Young conjugate is again $A$. 
One has that
\begin{equation} \label{eq:YoungCompl}
    t\le A^{-1}(t)\,\widetilde{A}^{-1}(t)\le 2t\quad \text{for $t\ge 0$,}
\end{equation}
where $A^{-1}$ denotes the generalized right-continuous inverse of $A$. 
A Young function $A$ is said to dominate another Young function $B$ if there
exists a constant $c>0$ such that
\begin{equation*}
	B(t) \leq A(ct)
	\quad \text{for $t \geq 0$.}
\end{equation*}
The functions $A$ and $B$ are called equivalent if they dominate each other. 
A Young function $A$ is said to essentially dominate another Young function $B$ if
\begin{equation*}
	 \sup_{0<t<\infty} \frac{A(t)}{B(\lambda t)} = \infty
	\quad\text{for every $\lambda\ge 1$.}
\end{equation*}

\subsection{Boyd indices}
Given a Young function $A$, we define the function $h_A\colon (0, \infty) \to [0, \infty)$ as 
\begin{equation*}
	h_A (t) = \sup_{0<s<\infty} \frac{A^{-1}(st)}{A^{-1}(s)} 
    \quad \text{for $t>0$.}
\end{equation*}
The lower and upper Boyd indices of $A$ are then defined as
\begin{equation*} 
    i_A = \sup_{1<t<\infty} \frac{\log t}{\log h_A (t)}
        \quad\text{and}\quad
    I_A	= \inf_{0<t<1} \frac{\log t}{\log h_A (t)}\,,
\end{equation*}
respectively. One has that
\begin{equation*}
    1\le i_A \le I_A \le \infty.
\end{equation*}
It can also be shown that
\begin{equation} \label{BIdef}
    i_A	= \lim_{t\to\infty} \frac{\log t}{\log h_A (t)}
        \quad\text{and}\quad
    I_A	= \lim_{t\to 0^+} \frac{\log t}{\log h_A (t)}.
\end{equation}
Further details on Boyd indices can be found for instance in \citep{Boy:71}.

\subsection{Orlicz spaces} \label{subsec:Orlicz}

Let $n\in\N$.
Denote by $\mathcal M(\Rn)$ the space of real-valued Lebesgue measurable
functions in $\Rn$. Given a 
Young function $A$, the Orlicz space $L^A(\Rn)$ is the collection  of all functions $f\in\mathcal{M}(\Rn)$ such that
\begin{equation*}
    \int_{\Rn} A\left(\frac{|f(x)|}{\lambda}\right)\,\d x <\infty
\end{equation*}
for some $\lambda >0$.
The Orlicz space $L^A(\Rn)$ is a Banach space endowed with the Luxemburg norm defined as 
\begin{equation*}
	\|f\|_{L^A(\Rn)} = \inf\left\{\lambda>0:
		\int_{\Rn}A\left(\frac{|f(x)|}{\lambda}\right)\,\d x \le 1\right\}
\end{equation*}
for $f \in \mathcal M(\Rn)$.
The choice  $A(t)=t^p$, with  $1\le p < \infty$, yields $L^A(\Rn) = L^p(\Rn)$,
the customary Lebesgue space. When $A(t)=0$ for $t \in [0,1]$ and $A(t)=\infty$
for $t \in (1, \infty)$, one has that $L^A(\Rn)=L^\infty (\Rn)$.

Orlicz space is an example of a so-called rearrangement-invariant space -- see
e.g. \cite[Section 4, Chapter 2]{Ben:88}.
Hence, given any Young function $A$, one has that 
\begin{equation*}
	\|f\|_{L^A(\Rn)} = \|f^*\|_{L^A(0,\infty)}
\end{equation*}
where $f^*$ is the nonincreasing rearrangement of $f$ given by
\[
	f^*(t)=\inf
	\bigl\{\lambda>0:
			\bigl|\bigl\{x\in\Rn: |f(x)|>\lambda\bigr\}\bigr|\le t
	\bigr\}\quad\text{for $t \geq 0$.}
\]

The inclusion relations between Orlicz spaces can be characterized in terms of
the notion of domination between Young functions. Let $A$ and $B$ be Young functions. Then 
\[
	L^{A}(\Rn) \to L^{B}(\Rn)
	\text{ if and only if $A$ dominates $B$}
\]
and the inclusion is strict
if and only if $A$ dominates $B$ essentially.

The alternate notation $A(L)(\Rn)$ for the Orlicz space $L^A(\Rn)$ will
be adopted when convenient. In particular, if
$A(t)$ is equivalent to $t^p\,\ell(t)^{\alpha_0}$ near zero
and to $t^p\,\ell(t)^{\alpha_\infty}$ near infinity
where
$\ell(t) = 1+|\log t|$, $t>0$,
and either $p>1$ and $\alpha_0$, $\alpha_\infty\in\R$
or $p=1$ and $\alpha_0 \le 0$ and $\alpha_\infty \ge 0$,
then the Orlicz space $L^A(\Rn)$ is the so called Zygmund space
denoted by $L^p(\log L)^{\mathbb A}(\Rn)$, where $\mathbb A = [\alpha_0,\alpha_\infty]$.
Observe that $i_A=I_A=p$.
If
\[
	A(t)
	\text{ is equivalent to }
	\begin{cases}
		\exp(-t^{\beta_0})
			& \text{near zero},
			\\
		\exp (t^{\beta_\infty})
			& \text{near infinity},
	\end{cases}
\]
with $\beta_0<0<\beta_\infty$, then Orlicz space $L^A(\Rn)$ is the space of exponential type,
denoted by $\exp L^{\mathbb B}(\Rn)$ where $\mathbb B=[\beta_0,\beta_\infty]$.
We have $i_A=I_A=\infty$ in this case.

In some particular, rather rare, situations, we adopt the relation
$L \simeq R$ if $L$ and $R$ bound each other up to
positive multiplicative constants independent of the quantities involved in
$L$ and $R$.

\section{Main results} \label{sec:main}

\noindent
We start by introducing a crucial tool
often named as a reduction principle. Such a principle
translates the boundedness of the operator $M_\gamma$ acting between
Orlicz spaces in $\Rn$ to a much simpler one-dimensional inequality
containing only the Young functions and the parameters $n$ and $\gamma$.
That enables us to simplify our analysis and
helps us to understand the behaviour of the operator
and the spaces involved.

At first, we need to introduce several constructions
of Young functions. Their importance will then be resembled in the
following theorem.
Let $A$ be a given Young function.
Assume that 
\begin{equation} \label{eq:Acond}
	\inf _{0<t<1} A(t)\, t^{-\ng} >0
\end{equation}
and define $\Ang$ by 
\begin{equation} \label{eq:AngDef}
	\Ang(t) = \int _0^t \frac { \Gngi(s)}{s} \d s
    	\quad \text{for $t \geq 0$,}
\end{equation}
where $\Gng \colon [0, \infty ) \to [0, \infty)$ is a nondecreasing function defined by
\begin{equation*}
	\Gng (t) = \sup_{0<s\leq t} A^{-1}(s) \,s^{-\gn}
		\quad \text{for $t \geq 0$}
\end{equation*}
and $\Gngi$ represents its generalized right-continuous inverse.

\begin{remark} \label{rem:AngYoung}
Note that $\Ang$ is a Young function.
First of all, \eqref{eq:Acond} ensures that $\Gng$ is
finite near zero and $\Gng$ is well-defined.
Next, observe that
\begin{equation*}
	\frac{\Gng(t)}{t}
		= t^{-\gn} \sup_{0<y<1} \frac{A^{-1}(ty)}{ty}y^{1-\gn}.
\end{equation*}
Since $A(t)/t$ is nondecreasing, $A^{-1}(s)/s$ is nonincreasing
and thus $\Gng(t)/t$ is decreasing and $\Gngi(s)/s$ is
increasing. Therefore $\Ang$ is a Young function and moreover,
by \eqref{eq:A},
$\Ang$ is globally equivalent to $\Gngi$.
\end{remark}

Let $B$ be a given Young function.
Assume that
\begin{equation} \label{eq:Bconvergeat0}
	\int_{0} \frac{B(s)}{s^{\ngnd+1}} \,\d s < \infty,
\end{equation}
and define (with the little abuse of notation)
\begin{equation} \label{eq:BngDef}
	\Bng(t) = \int_{0}^{t} \frac{\Engi(s)}{s}\,\d s
    	\quad \text{for $t \geq 0$.}
\end{equation}
Here, $\Eng\colon[0,\infty) \to [0,\infty)$ is given by
\[
	\Eng(t) = t^\gn \Fngi(t)
    	\quad \text{for $t \geq 0$,}
\]
where $\Fng\colon[0,\infty) \to [0,\infty)$ is defined by
\begin{equation} \label{eq:FngDef}
	\Fng(t) = t^{\ngnd} \int_{0}^{t} \frac{B(s)}{s^{\ngnd+1}} \,\d s
    	\quad \text{for $t \geq 0$.}
\end{equation}

\begin{remark}
Observe that $\Bng$ is well-defined Young function.
The assumption \eqref{eq:Bconvergeat0} guarantees that
$\Fng$ is finite and increasing. Thus $\Fngi$ is increasing
and so is $\Eng$. Next, since $\Fng(t)/t^\ngnd$ is nondecreasing,
$\Fngi(t)/t^{1-\gn}$ is nonincreasing, hence
$\Eng(t)/t$ is nonincreasing and $\Engi(s)/s$ is nondecreasing.
Therefore, $\Bng$ is convex.
In addition, by \eqref{eq:A}, $\Bng$ is  equivalent to $\Engi$.
\end{remark}

Our first principal result then reads as follows.

\begin{theorem}[Reduction principle in Orlicz spaces]
\label{thm:Orlicz_reduction}
Let $n\in\N$ and $0<\gamma<n$ and suppose that $A$ and $B$ are Young functions.
The following assertions are equivalent:
\begin{enumerate}
\item There exists a constant $C_1>0$ such that
\[
	\| M_\gamma f \|_{L^B(\Rn)} \le C_1 \|f\|_{L^A(\Rn)}
\]
for every $f\in L^A(\Rn)$;
\item There exists a constant $C_2>0$ such that
\[
	\int_{\Rn} B \biggl( \frac{M_\gamma f(x)}{C_2 \bigl( \int_{\Rn} A( |f(y)| )\,\d y\bigr)^\gnd } \biggr) \,\d x
		\le \int_{\Rn} A\bigl( |f(x)| \bigr)\,\d x
\]
for every $f\in L^A(\Rn)$;
\item $A$ satisfies \eqref{eq:Acond}
and there exists a constant $C_3>0$ such that
\[
	\int_{0}^{t} \frac{B(s)}{s^{\ngnd+1}}\,\d s
		\le \frac{\Ang(C_3 t)}{t^\ngnd}
		\quad\text{for $t>0$}
\]
where $\Ang$ is the Young function from \eqref{eq:AngDef};
\item $B$ satisfies \eqref{eq:Bconvergeat0}
and there exists a constant $C_4>0$ such that
\[
	\Bng(t) \le A(C_4t)
		\quad\text{for $t>0$}
\]
where $\Bng$ is the Young function from \eqref{eq:BngDef}.
\end{enumerate}
Moreover, the constants $C_1$, $C_2$, $C_3$ and $C_4$ depend on each other and on $n$ and $\gamma$.
\end{theorem}

We would like to point out the philosophy behind the criteria (iii) and (iv)
in our reduction principle. They look completely different at a first glance
and they rely on auxiliary Young functions $A_\ng$ or $B_\ng$ respectively.
In situations when both $A$ and $B$ are given, there is no significantly better
choice of a condition to check. However, imagine that we have one $A$ and
a bunch of candidates $B$ to choose from. Then the condition (iii) comes handy
as we compute $A_\ng$ once and we check the inequality against every choice of
$B$. The condition (iv) is then welcome in the reciprocal case.

It is no surprise that the Young functions $A_\ng$ and $B_\ng$
play the major role in the problem of establishing the corresponding
optimal Orlicz spaces. Let us begin with the targets.

\begin{theorem}[Optimal Orlicz target] \label{thm:OOT}
Let $n\in\N$ and $0<\gamma<n$. Suppose that $A$ is a Young function
satisfying \eqref{eq:Acond} and set $\Ang$ as in \eqref{eq:AngDef}.
If
\begin{equation} \label{eq:Angindex}
	i_{\Ang} > \ngn,
\end{equation}
then
\begin{equation} \label{eq:MgAAng}
	M_\gamma\colon L^A(\Rn) \to L^\Ang(\Rn)
\end{equation}
and $L^\Ang(\Rn)$ is the smallest among all Orlicz spaces in \eqref{eq:MgAAng}.

Conversely, if \eqref{eq:Angindex} fails, then
there is no optimal Orlicz space in \eqref{eq:MgAB} in a sense
that any Orlicz space $L^B(\Rn)$ for which \eqref{eq:MgAB} holds true can be
replaced by a strictly smaller Orlicz space for which \eqref{eq:MgAB} is still valid.

In particular, if $I_A < \ng$, then \eqref{eq:Angindex} is equivalent to $i_A > 1$ and
\begin{equation} \label{eq:supout}
	\Angi(t) \simeq A^{-1}(t)\, t^{-\gn},
		\quad t>0.
\end{equation}

In addition, if \eqref{eq:Acond} is not satisfied, then there does not exist an Orlicz target space $L^B(\Rn)$
for which \eqref{eq:MgAB} holds.
\end{theorem}

\def\pcon#1{{\frac{np_{#1}}{n-\gamma p_{#1}}}}
\def\acon#1{{\frac{n\alpha_{#1}}{n-\gamma p_{#1}}}}
\begin{example} \label{ex:LZrange}
Let $1\le p_0, p_\infty\le \infty$ and $\alpha_0, \alpha_\infty\in\R$.
If $p_0=1$ then let $\alpha_0\le 0$ and if $p_\infty=1$ then let $\alpha_\infty \ge 0$.
Suppose that
\[
	A(t)
	\text{ is equivalent to }
	\begin{cases}
		t^{p_0}\,\ell(t)^{\alpha_0}
			&\text{ near zero},
			\\
		t^{p_\infty}\,\ell(t)^{\alpha_\infty}
			&\text{ near infinity}.
	\end{cases}
\]
Let us consider the nontrivial cases only, i.e.\ let us assume
that \eqref{eq:Acond} is satisfied. This implies that either $1\le p_0 <\ng$
or $p_0=\ng$ and $\alpha_0 \ge 0$.
Computations show that
\[
	A_\ng(t)
	\text{ is equivalent to }
	\begin{cases}
		\vtop{\vss\hbox{$t^\pcon0 \ell (t)^{\acon0},$}\vss}
			& \vtop{\lineskip=\jot%
				\hbox{$1\le p_0 < \ng$, $\alpha_0 \in \R$ or}
				\hbox{$p_0=1$, $\alpha_0 \le 0$,}}
			\\
		\exp(-t^{-\frac{n}{\gamma\alpha_0}}),
			& \text{$p_0 = \ng$, $\alpha_0 > 0$,}
			\\
		0,
			& \text{$p_0 = \ng$, $\alpha_0 = 0$,}
	\end{cases}
\]
near zero and
\[
	A_\ng(t)
	\text{ is equivalent to }
	\begin{cases}
		\vtop{\vss\hbox{$t^\pcon\infty \ell (t)^{\acon\infty}$,}\vss}
			& \vtop{\lineskip=\jot%
				\hbox{$1\le p_\infty < \ng$, $\alpha_\infty \in \R$ or}
				\hbox{$p_\infty=1$, $\alpha_\infty\ge 0$,}}
			\\
		\exp(t^{-\frac{n}{\gamma\alpha_\infty}}),
			& \text{if $p_\infty = \ng$, $\alpha_\infty < 0$,}
			\\
		\vtop{\vss\hbox{$\infty,$}\vss}
			&
			\vtop{\lineskip=\jot%
			\hbox{$p_\infty = \ng$, $\alpha_\infty \ge 0$ or}
			\hbox{$p_\infty > \ng$, $\alpha_\infty\in\R$,}}
	\end{cases}
\]
near infinity.
Moreover,
\[
	i_{A_\ng} = 
	\begin{cases}
		\min\left\{\pcon0,\pcon\infty  \right\},
			& 1\le \min\{p_0,p_\infty\} < \ng,
			\\
		\infty,
			& p_0=p_\infty=\ng,
	\end{cases}
\]
whence $i_{A_\ng} > \ngn$ if and only if both $p_0>1$ and $p_\infty>1$.
Therefore, by Theorem~\ref{thm:OOT},
\[
	M_\gamma\colon L^A(\Rn) \to L^{A_\ng}(\Rn)
\]
and the range spaces are optimal among all Orlicz spaces.

If $p_0=1$ or $p_\infty=1$, then every Young function $B$ satisfying \eqref{eq:MgAB}
can be essentially enlarged near zero or near infinity, respectively.
\end{example}

\begin{example}
Let $1\le p_0, p_\infty\le\infty$.
Suppose that
\[
	A(t)
	\text{ is equivalent to }
	\begin{cases}
		t^{p_0}\,e^{-\sqrt{\log 1/t}}
			&\text{ near zero},
			\\
		t^{p_\infty}\,e^{\sqrt{\log t}}
			&\text{ near infinity}.
	\end{cases}
\]
In order to ensure \eqref{eq:Acond}, assume that $1\le p_0 < \ng$.
We have that $A(t)t^{-\ng}$ is decreasing near zero hence
$A^{-1}(s)s^{-\gn}$ is increasing near zero. Thus the supremum
in the definition of $\Ang$ may be disregarded and $\Ang$ obeys
\eqref{eq:supout} near zero. Calculation shows that
\[
	A^{-1}(s) \simeq s^\frac{1}{p_0} e^{-p_0^{-\frac32}\sqrt{\log 1/s}}
	\quad\text{near zero}
\]
and then
\[
	\Ang(t)
		\quad\text{is equivalent to}\quad
		t^{\pcon0} e^{-(\frac{n}{n-\gamma p_0})^{\frac32}\sqrt{\log 1/t}}
	\quad\text{near zero}
\]
If $p_\infty\ge\ng$, then $A(t)t^{-\ng}$ is increasing,
$A^{-1}(s)s^{-\gn}$ is decreasing, hence $\Gng$ is a constant function
and $\Ang=\infty$ near infinity.
If $1\le p_\infty < \ng$, then, similarly as before, $\Ang$ satisfies
\eqref{eq:supout} near infinity. To sum it up,
\[
	\Ang(t)
		\text{ is equivalent to }
		\begin{cases}
			t^{\pcon\infty} e^{(\frac{n}{n-\gamma p_\infty})^{\frac32}\sqrt{\log t}},
				& 1\le p_\infty < \ng,
				\\
			\infty,
				& p_\infty \ge \ng,
		\end{cases}
\]
near infinity. In conclusion
\[
	i_{A_\ng} = 
	\begin{cases}
		\min\left\{\pcon0,\pcon\infty  \right\},
			& 1\le p_\infty < \ng,
			\\
		\pcon0,
			& p_\infty\ge\ng,
	\end{cases}
\]
and $i_\Ang > \ngn$ if and only if both $p_0>1$ and $p_\infty>1$,
otherwise any Young function $B$ satisfying \eqref{eq:MgAB} can be essentially
enlarged near zero or near infinity, respectively.
\end{example}

\begin{theorem}[Optimal Orlicz domain] \label{thm:OOD}
Let $n\in\N$ and $0<\gamma<n$
and suppose that $B$ is a Young function. If $B$ satisfies \eqref{eq:Bconvergeat0},
then $B_\ng$ is a Young function for which
\begin{equation} \label{eq:MgBngB}
	M_\gamma\colon L^\Bng(\Rn) \to L^B(\Rn)
\end{equation}
and $L^\Bng(\Rn)$ is the largest possible Orlicz space satisfying
\eqref{eq:MgBngB}.

In addition, if $i_B > \ngn$, then $B$ satisfies \eqref{eq:Bconvergeat0}
and $B_\ng$ obeys the simpler relation
\begin{equation} \label{eq:intout}
	B_\ng^{-1}(t) \simeq t^\gn\,B^{-1}(t),
		\quad t>0.
\end{equation}

Conversely, if \eqref{eq:Bconvergeat0} fails, then
there is no Orlicz space $L^A(\Rn)$ for which \eqref{eq:MgAB} holds.
\end{theorem}

\def\qcon#1{{\frac{nq_{#1}}{n+\gamma q_{#1}}}}
\def\acon#1{{\frac{n\alpha_{#1}}{n+\gamma q_{#1}}}}
\begin{example} \label{ex:LZdomain}
Let $1 < q_0,q_\infty\le \infty$ and $\alpha_0,\alpha_\infty\in\R$. Suppose that
$B$ is a Young function such that
\[
	B(t)
	\text{ is equivalent to }
	\begin{cases}
		t^{q_0}\,\ell(t)^{\alpha_0}
			&\text{ near zero},
			\\
		t^{q_\infty}\,\ell(t)^{\alpha_\infty}
			&\text{ near infinity}.
	\end{cases}
\]
The condition \eqref{eq:Bconvergeat0} requires that
either $q_0 > \ngnd$ or $q_0=\ngnd$ and $\alpha_0 < -1$.
One can compute that
\[
	B_\ng(t)
	\text{ is equivalent to }
	\begin{cases}
		t^\qcon0\,\ell(t)^\acon0,
			& q_0 > \ngn,
			\\
		t\,\ell(t)^{(1-\gn)(\alpha_0+1)},
			& q_0 = \ngn,\, \alpha_0 < -1,
	\end{cases}
\]
near zero and
\[
	B_\ng(t)
	\text{ is equivalent to }
	\begin{cases}
		t^\qcon\infty\,\ell(t)^\acon\infty,
			& q_\infty > \ngn,
			\\
		t\,\ell(t)^{(1-\gn)(\alpha_\infty+1)},
			& q_\infty = \ngn,\, \alpha_\infty > -1,
			\\
		t\,\ell(\ell(t))^{(1-\gn)},
			& q_\infty = \ngn,\, \alpha_\infty = -1,
			\\
		\vtop{\vss\hbox{$t,$}\vss}
			&
			\vtop{\lineskip=\jot%
			\hbox{$q_\infty = \ngn$, $\alpha_\infty < -1$ or}
			\hbox{$q_\infty < \ng$,}}
	\end{cases}
\]
near infinity. By Theorem~\ref{thm:OOD},
\[
	M_\gamma\colon L^{B_\ng}(\Rn) \to L^B(\Rn)
\]
and $L^{B_\ng}(\Rn)$ is the optimal Orlicz space.
\end{example}

\begin{example}
Let $1 \le q_0,q_\infty\le \infty$. We deal with the Young
function $B$ such that
\[
	B(t)
	\text{ is equivalent to }
	\begin{cases}
		t^{q_0}\,e^{-\sqrt{\log 1/t}}
			&\text{ near zero},
			\\
		t^{q_\infty}\,e^{\sqrt{\log t}}
			&\text{ near infinity}.
	\end{cases}
\]
The condition \eqref{eq:Bconvergeat0} forces
that $q_0\ge \ngnd$.
If $q_0 > \ngnd$ then $i_B > \ngnd$
and $\Bng$ satisfies the simplified relation \eqref{eq:intout} near
zero.
In the case when $q_0=\ngnd$ then
\[
	\Fng(t) \simeq t^\ngn e^{-\sqrt{\log 1/t}}\sqrt{\log 1/t}
	\quad\text{near zero}
\]
and
\[
	\Bngi(s) \simeq \Fngi(s) s^\gn
		\simeq se^{-(1-\gn)^\frac32\sqrt{\log 1/s}} (\log 1/s)^{-\frac{n-\gamma}{2n}}
	\quad\text{near zero}.
\]
Calculating the inverses, we get that
\[
	\Bng(t)
	\text{ is equivalent to }
	\begin{cases}
		t^{\qcon0} e^{-(\frac{n}{n+\gamma q_0})^{\frac32}\sqrt{\log 1/t}},
		& q_0 > \ngn,
			\\
		t e^{-(1-\ng)^\frac32 \sqrt{\log 1/t}} (\log 1/t)^\frac{2n}{n-\gamma},
		& q_0 = \ngn,
	\end{cases}
\]
near zero.

Let us also sketch the calculations near infinity. If $g_\infty< \ngnd$,
then
\[
	\int^{\infty} \frac{B(s)}{s^{\ngnd+1}}\,\d s <\infty
\]
whence $\Fng(t) \simeq t^\ngn$ and $\Bng(t)$ is equivalent to $t$ near
infinity. If $q_\infty >\ngnd$ then $\Bng$ obeys the relation \eqref{eq:intout}
near infinity, and finally, when $q_\infty=\ngnd$, then
\[
	\Fng(t) \simeq t^\ngn e^{\sqrt{\log t}}\sqrt{\log t}
	\quad\text{near zero}.
\]
In conclusion, we have
\[
	\Bng(t)
	\text{ is equivalent to }
	\begin{cases}
		t^{\qcon\infty} e^{(\frac{n}{n+\gamma q_\infty})^{\frac32}\sqrt{\log t}},
		& q_\infty > \ngn,
			\\
		t e^{(1-\ng)^\frac32 \sqrt{\log t}} (\log t)^\frac{2n}{n-\gamma},
		& q_\infty = \ngn,
			\\
		t,
		& q_\infty < \ngn,
	\end{cases}
\]
near infinity and by Theorem~\ref{thm:OOD},
\[
	M_\gamma\colon L^{B_\ng}(\Rn) \to L^B(\Rn),
\]
in which $L^{B_\ng}(\Rn)$ is optimal within Orlicz spaces.
\end{example}

Concerning optimality, one may naturally ask the question
if the relation ``be optimal Orlicz space for someone''
is symmetric. We will look closely what is meant by this now.

Let us start on the target side, so let us have some Young
function $B$ fixed. By Theorem~\ref{thm:OOD}, the optimal Orlicz domain
always exists and is described by the Young function $B_\ng$.
At this stage we may set $A=B_\ng$ and try to use Theorem~\ref{thm:OOT}
to investigate the corresponding best possible Orlicz target.

We illustrate what is happening on a basic example.
Assume that $B(t) = t^q$ near infinity
and $q>\ngn$
(we will be ignoring the behaviour near zero for the sake of this paragraph,
the careful reader may adapt the Young functions also near zero).
Then $B_\ng(t)$ is equivalent to $t^{\frac{nq}{n+\gamma q}}$ near infinity.
Set $A=B_\ng$ and observe that $A_\ng(t)$ is equivalent
to $t^q$, whence $i_{A_\ng}=q>\ngn$. Thus $A_\ng$ is equivalent to
$B$ and both domain $L^A$ and range $L^B$ are optimal Orlicz spaces.

Now, if $B(t)=t^\ngn$ near infinity, then
$B_\ng$ is equivalent to $t\log(t)^{(1-\ng)}$ near infinity.
Denoting $A=B_\ng$, $A_\ng(t)$ is equivalent to $t^\ngnd\log(t)$
which exceeds~$B(t)$. However, $i_{A_\ng}=\ngn$ thus $B$ is not
the optimal Orlicz space to $L^A$ and moreover, no such Orlicz
space exists.

From this example we may guess that the borderline lays somewhere
around the space $L^\ngn$. Indeed, the proper classification
of this phenomenon relies on the Boyd index of $B$ as the following
theorem shows.

\begin{theorem}[Orlicz range reiteration]\label{thm:ORR}
Let $B$ be a Young function. Suppose that $B$ obeys
\begin{equation} \label{eq:Bindex}
	i_B > \ngn.
\end{equation}
Then the Young function $B_\ng$ from \eqref{eq:BngDef}
satisfies \eqref{eq:intout}
and
\begin{equation} \label{eq:MgBngBopt}
	M_\gamma\colon L^{B_\ng}(\Rn) \to L^B(\Rn)
\end{equation}
where both domain and target spaces are optimal among all Orlicz spaces.

Conversely, if \eqref{eq:Bindex} fails, then
$L^{B_\ng}(\Rn)$ is optimal Orlicz domain and no optimal Orlicz range exists
in \eqref{eq:MgBngBopt}.
\end{theorem}

In this iteration scheme, we may also assume that $A$ is given and we try
to make one step further and then one step back, or more precisely,
we can compute $A_\ng$, then set $B=A_\ng$ and then analyze the
relation of $B_\ng$ and $A$. The main difference between this case
and the previous one is that even the success after first step is not
guaranteed any more. So one has to restrict his attention to the positive
cases only. Let us look at similar trivial example.

Assume that $A(t)=t^{p_0}$ near zero and $A(t)=t^{p_\infty}$ near infinity,
where $1<p_0<\ng<p_\infty<\infty$. Calculations gives that $A_\ng(t)$
is equivalent to $t^{\frac{np_0}{n-\gamma p_0}}$ and to $\infty$ near infinity.
Since $i_{A_\ng} > \ngn$, $M_\gamma$ acts from $L^A$ to $L^{A_\ng}$ and the range is
optimal within Orlicz spaces. If we set $B=A_\ng$, then $B_\ng(t)$ is equivalent
to $t^{p_0}$ near zero and to $t^\ng$ near infinity. We see that
$B_\ng$ coincides with $A$ near zero while there is a significant improvement
near infinity.

To state the result in its full generality, we need to introduce
a way how to define the improved Young function to $A$.

Let $A$ be a Young function satisfying \eqref{eq:Acond}
and let $\Asup$ be given by
\begin{equation} \label{eq:Asupdef}
	\Asup(t) = \int_{0}^{t} \frac{\Gsupi(s)}{s}\,\d s, 
		\quad t > 0.
\end{equation}
where $\Gsup$ is defined by
\[
	\Gsup(t) = t^\ng \sup_{0<s\le t} A^{-1}(s)\, t^{-\ng},
		\quad t > 0.
\]
Using similar arguments as in Remark~\ref{rem:AngYoung},
one can easily observe that $\Gsup$ is increasing and
$\Asup$ is well-defined Young function equivalent to
$\Gsupi$.

The spirit of this improvement lays in the observation
that any domain $A$ can be
always replaced by $\Asup$. This is the essence of the next theorem.

\begin{theorem}[Orlicz domain reiteration] \label{thm:ODR}
Let $A$ and $B$ be a Young functions.
Suppose that $A$ satisfies \eqref{eq:Acond}
and let $\Asup$ be the Young function from \eqref{eq:Asupdef}.
Then \eqref{eq:MgAB} holds
if and only if
\begin{equation} \label{eq:MgAsupB}
	M_\gamma\colon L^\Asup(\Rn) \to L^B(\Rn).
\end{equation}

Moreover, if
\begin{equation} \label{eq:Asupindex}
	i_\Asup > 1,
\end{equation}
then
\[
	M_\gamma\colon L^\Asup(\Rn) \to L^{A_\ng}(\Rn)
\]
and both domain and target spaces are optimal in the class of Orlicz spaces.

Conversely, if \eqref{eq:Asupindex} fails, then
$L^A(\Rn)$ can be replaced by $L^\Asup(\Rn)$
and no optimal Orlicz target exists.
\end{theorem}

At the end of this section, we present special cases
of the reduction principle for the spaces $L^1$ and $L^\infty$.

\begin{corollary}[Endpoint embeddings] \label{cor:endpoints} 
Let $n\in\N$ and $0<\gamma<n$ and suppose that $A$ and $B$ are Young functions.
Then the following statements hold true:
\begin{enumerate}
\item
\begin{equation*}
	M_\gamma\colon L^A(\Rn) \to L^\infty(\Rn)
\end{equation*}
if and only if there is a constant $C>0$ such that
\begin{equation*} 
	A(t) \ge C t^\ng
		\quad\text{for $t>0$};
\end{equation*}
\item 
\begin{equation*} 
	M_\gamma\colon L^1(\Rn) \to L^B(\Rn)
\end{equation*}
if and only if
\begin{equation*} 
	\int_{0}^\infty \frac{B(s)}{s^{\ngnd+1}}\,\d s < \infty.
\end{equation*}
\end{enumerate}
\end{corollary}

\section{Proofs} \label{sec:proofs}

We start with an auxiliary reduction principle for fractional maximal operator.
Note that the result remains valid also if we replace Orlicz
spaces by any so-called rearrangement-invariant function spaces.
For further details on this general setting we refer to \citep{Edm:18}.

\begin{proposition} \label{prop:ri_reduction}
Let $n\in\N$, $0<\gamma<n$ and
let $A$ and $B$ be Young functions.
The following statements are equivalent:
\begin{enumerate}
\item 
There is a constant $C_1>0$ such that
\[
	\|M_\gamma f\|_{L^B(\Rn)} \le C_1 \|f\|_{L^A(\Rn)},
		\quad f\in L^A(\Rn);
\]
\item There is a constant $C_2>0$ such that
\begin{equation*}
	\biggl\| t^{\gn -1}\int_{0}^{t} g(s)\,\d s \biggr\|_{L^B(0,\infty)}
		\le C_2 \|g\|_{L^A(0,\infty)},
		\quad g\in L^A(0,\infty).
\end{equation*}
\end{enumerate}
Moreover, the constants $C_1$ and $C_2$ depend on each other, on $n$ and $\gamma$.
\end{proposition}

\begin{proof}
Assume (ii).
By \citep[Theorem~1.1]{Cia:00a}, there is a constant $c_1>0$,
depending only on $n$ and $\gamma$,
such that
\[
	(M_\gamma f)^*(t) \le c_1 \sup_{t\le s <\infty} s^\gn f^{**}(s)
		\quad\text{for $t>0$}
\]
and for every $f\in \Liloc(\Rn)$.
Here, $f^{**}$ is the function defined by
\[
	f^{**}(t) = \frac{1}{t}\int_{0}^{t} f^*(s)\,\d s
		\quad\text{for $t>0$}.
\]
Then
\[
	\|M_\gamma f\|_{L^B(\Rn)}
		= \|(M_\gamma f)^* \|_{L^B(0,\infty)}
		\le c_1 \Bigl\| \sup_{t\le s <\infty} s^\gn f^{**}(s)
		\Bigr\|_{L^B(0,\infty)}.
\]
The supremum may be dropped paying another constant $c_2=c_2(n,\gamma)$, thanks
to \citep[Theorem~3.9]{Ker:06}. Hence
\[
	\|M_\gamma f\|_{L^B(\Rn)}
		\le c_1 c_2 \| t^\gn f^{**}(t) \|_{L^B(0,\infty)}
		\le c_1 c_2 C_2 \|f^*\|_{L^A(0,\infty)}
		= C_1 \|f\|_{L^A(\Rn)},
\]
due to the assumption (ii) where we take $g=f^*$.

Conversely, suppose (i). Let $\varphi\colon(0,\infty)\to[0,\infty)$ be 
a nonincreasing function. By \citep[Theorem~1.1]{Cia:00a},
there is a function $f$ on $\Rn$ such that $f^*=\varphi$ and
\[
	(M_\gamma f)^*(t) \ge c_3 \sup_{t\le s <\infty} s^\gn f^{**}(s)
		\quad\text{for $t\in(0,\infty)$}
\]
where $c_3=c_3(n,\gamma)$ is a positive constant independent of $f$.
We have
\begin{multline*}
	C_1 \|\varphi\|_{L^A(0,\infty)}
		= C_1 \|f\|_{L^A(\Rn)}
		\ge \|M_\gamma f\|_{L^B(\Rn)}
		= \|(M_\gamma f)^*\|_{L^B(0,\infty)}
		\\
		\ge c_3 \Bigl\|
			\sup_{t\le s <\infty} s^\gn f^{**}(s) \Bigr\|_{L^A(0,\infty)}
		\ge c_3 \biggl\|
			t^{\gn-1} \int_{0}^{t}\varphi(s)\,\d s \biggr\|_{L^A(0,\infty)},
\end{multline*}
hence (ii) holds for every nonincreasing function $\varphi$ with
$C_2=C_1/c_3$. The inequality (ii) for any function then
follows by Hardy's lemma.
\end{proof}

\begin{proof}[Proof of Theorem~\ref{thm:Orlicz_reduction}]
We begin with the preliminary statement, equivalent to~(i).
Recall the Hardy operator $H_\gn$, defined by
\[
	H_\gn f(t) = \int_{t}^{\infty} f(s)s^{\gn-1}\,\d s,
		\quad t\in (0,\infty),
\]
for $f\in\mathcal M(0,\infty)$ and its dual, $H_\gn'$ say, given by
\[
	H_\gn' g(t) = t^{\gn-1}\int_{0}^{t} g(s)\,\d s,
		\quad t\in (0,\infty),
\]
for $g\in\mathcal M(0,\infty)$.
By Proposition~\ref{prop:ri_reduction}, (i) holds if and only if
$H_\gn'$ is bounded from $L^A(0,\infty)$ to $L^B(0,\infty)$
with the operator norm comparable to $C_1$. Hence, by the duality,
(i) is equivalent to the following statement.
\begin{enumerate}
\item[\rm (i')]
There exist a constant $C_1'>0$ such that
\begin{equation*}
	\| H_\gn f \|_{L^{\tilde A}(0,\infty)} \le C_1' \|f\|_{L^{\tilde B}(0,\infty)}.
\end{equation*}
\end{enumerate}
Moreover the constants $C_1$ and $C_1'$ are comparable.

We now show (i')$\Leftrightarrow$(iii). By the combination
of \citep[Proposition~5.2 and Lemma~5.3]{Cia:18},
(i') holds true if and only if
\begin{equation} \label{feb2}
	\sup_{0<t<1} \tilde{A}(t)\,t^\ngn <\infty
\end{equation}
and
there is a constant $C_3'>0$ such that
the inequality
\begin{equation*}
	\int_{0}^{t} \frac{B(s)}{s^{\ngnd+1}}\,\d s
		\le \frac{E(C_3't)}{t^\ngnd},
		\quad t>0,
\end{equation*}
holds true, where $E$ is a Young function satisfying
\begin{equation*}
	\tilde{E}^{-1} (t)
		\simeq  t \inf_{0<s\le t}
			\tilde{A}^{-1}(s)\, s^{\gn - 1},
		\quad t>0.
\end{equation*}
Now, clearly \eqref{feb2} holds if and only if \eqref{eq:Acond} holds and $E$
is equivalent to $\Ang$, by passing to complementary Young functions using
\eqref{eq:YoungCompl}.

The equivalence of (i') and (iv) is a consequence of
a different method of characterization of the boundedness
of the Hardy operator $H_\gn$ in Orlicz spaces \citep[Lemma~1]{Cia:96} (see
also \citep[Section 9.3]{Rao:02}).
We moreover use the equivalent formulation of $B_\ng$ which might be
derived from \citep[(4.6) and (4.7)]{Cia:96}.

Let us establish (ii) from (i).
Suppose that $N>0$ is given and set $A_N=A/N$, a scaled Young function.
Then
\[
	(A_N)_\ng(t) = \frac{1}{N} A_\ng \bigl( tN^{-\gn} \bigr),
		\quad t\ge 0,
\]
where $(A_N)_\ng$ is a Young function associated to $A_N$ as in \eqref{eq:AngDef}.
Define also $B_N$ by
\[
	B_N(t) = \frac{1}{N} B\bigl( tN^{-\gn} \bigr),
		\quad t\ge 0.
\]
We claim that
\begin{equation} \label{feb4}
	\|M_\gamma f\|_{L^{B_N}(\Rn)} \le C_2 \|f\|_{L^{A_N}(\Rn)}
\end{equation}
for all $f\in L^{A_N}(\Rn)$ with the constant $C_2$ independent of $N$.
Indeed, as one can readily check by the change of variables, (iii) holds
with $A$ and $B$ replaced by $A_N$ and $B_N$ respectively with
the same constant $C_3$. The claim follows by the already proven equivalence
of (i) and (iii).

Now, let $f\in L^A(\Rn)$. If $\int_{\Rn} A(|f|) = \infty$, then
there is nothing to prove. Otherwise, set $N=\int_{\Rn} A(|f|)$.
It is $\|f\|_{L^{A_N}(\Rn)} \le 1$
and, by \eqref{feb4}, $\|f\|_{L^{B_N}(\Rn)} \le C_2$.
Therefore
\[
	\int_{\Rn} B_N \biggl( \frac{M_\gamma f(x)}{C_2} \biggr)\,\d x
		\le 1
\]
and (ii) follows by the definition of $B_N$.

The converse implication (ii)$\Rightarrow$(i) is trivial.
\end{proof}

\begin{proof}[Proof of Corollary~\ref{cor:endpoints}]
(i) Suppose that $B(t)=0$ on $[0,1]$ and $B(t)=\infty$ for $t>1$.
Then $L^B(\Rn) = L^\infty(\Rn)$, \eqref{eq:Bconvergeat0}
holds and, by \eqref{eq:BngDef},
$B_\ng$ is equivalent to $t^\ng$ on $[0,\infty)$.
By Theorem~\ref{thm:Orlicz_reduction} (the equivalence of (i) and (iv)),
we have (i).

(ii) Let us set $A(t)=t$, $t\ge 0$, so $L^A(\Rn)=L^1(\Rn)$.
Clearly, $A$ satisfies \eqref{eq:Acond} and, by \eqref{eq:AngDef},
$A_\ng(t)$ is equivalent to $t^\ngn$ on $[0,\infty)$.
Thus (ii) follows by Theorem~\ref{thm:Orlicz_reduction},
the equivalence of (i) and (iii).
\end{proof}

\begin{lemma} \label{lemm:inheritedind}
Let $B$ be a Young function satisfying \eqref{eq:Bconvergeat0} and
let $F_\ng$ be defined as in \eqref{eq:FngDef}.
Then  $F_\ng$ is a Young function and
\[
	i_{F_\ng} > \ngn
	\quad\text{if and only if}\quad
	i_B > \ngn.
\]
Furthermore, if this is the case, then $F_\ng$ is equivalent to $B$.
\end{lemma}

\begin{proof}
Suppose that $i_B > \ngn$. Then, by \citep[Proposition~4.1]{Cia:18},
$B$ is equivalent to $F_\ng$ and therefore $i_{F_\ng}>\ngn$.

Conversely, suppose that $i_{F_\ng} > \ngn$. Then, thanks to
\citep[Proposition 4.1]{Cia:18}, there exist $\sigma>1$ and $c>1$
such that
\begin{equation} \label{mar1}
	F_\ng (\sigma t) \ge c \sigma^\ngn F_\ng(t),
		\quad t>0.
\end{equation}
Let us write $c=1+\varepsilon$ for some $\varepsilon>0$. Then
\eqref{mar1} becomes
\begin{equation*}
	\int_0^{t}
			\frac{B(s)}{s^{\ngnd+1}}\d s
		\le \frac{1}{\varepsilon} \int_{t}^{\sigma t}
			\frac{B(s)}{s^{\ngnd+1}}\d s
		\le B (\sigma t) \frac{1}{\varepsilon} \int_{t}^{\sigma t}
			\frac{1}{s^{\ngnd+1}}\d s
		\le \frac{B(kt)}{t^\ngnd}
\end{equation*}
for every $t>0$ and for a suitable constant $k$.
Another use of \citep[Proposition 4.1]{Cia:18} gives that $i_B>\ngn$.
\end{proof}

\begin{proof}[Proof of Theorem~\ref{thm:OOD}]
Assume that $B$ satisfies \eqref{eq:Bconvergeat0}. Then $B_\ng$ is well-defined
and \eqref{eq:MgBngB} holds by Theorem~\ref{thm:Orlicz_reduction}.
To observe the optimality, assume that $L^A(\Rn)$ satisfies \eqref{eq:MgAB}. Then,
again due to Theorem~\ref{thm:Orlicz_reduction}, $A$ dominates $B_\ng$, whence
$L^{A}(\Rn)\to L^{B_\ng}(\Rn)$ and $L^{B_\ng}(\Rn)$ is optimal.

If \eqref{eq:Bconvergeat0} is violated, then no Orlicz space $L^A(\Rn)$ might
satisfy \eqref{eq:MgAB} since otherwise it would contradict 
Theorem~\ref{thm:Orlicz_reduction}.

Furthermore, if $i_B > \ngn$, then Lemma~\ref{lemm:inheritedind} gives
that $F_\ng$ is equivalent to $B$ and the formula \eqref{eq:intout} is immediate.
\end{proof}

\begin{lemma}\label{lemm:construction}
Let $n\in\N$, $0<\gamma<n$ and let $D$ be a Young function
such that $D(t)/t^\ngnd$ is nondecreasing,
\begin{equation} \label{eq:limD0}
	\lim_{t\to 0} \frac{D(t)}{t^\ngnd} = 0
\end{equation}
and
\begin{equation} \label{eq:supintD}
	\sup_{0<t<1} \frac{t^\ngnd}{D(Kt)} \int_{0}^{t} \frac{D(s)}{s^{\ngnd+1}}\,\d s
		= \infty
\end{equation}
for every $K\ge 1$.
Suppose that $B$ is a Young such that
\begin{equation} \label{eq:ineqDBzero}
	\int_{0}^{t} \frac{B(s)}{s^{\ngnd+1}}\,\d s
		\le \frac{D(Ct)}{t^\ngnd},
		\quad 0<t<1,
\end{equation}
for some $C\le 1$.
Then there exists a Young function $B_1$ essentially dominating $B$
and satisfying 
\begin{equation} \label{eq:ineqDB1zero}
	\int_{0}^{t} \frac{B_1(s)}{s^{\ngnd+1}}\,\d s
		\le \frac{D(5Ct)}{t^\ngnd},
		\quad 0<t<1.
\end{equation}
\end{lemma}

\begin{proof}
Let $D$ and $B$ be the Young functions from the statement.
Fix $t\in(0,1)$ and define the set $G_t$ by
\[
	G_t = \bigl\{ s\in(0,1): \tfrac{B(s)}{s} \le \tfrac{D(t)}{t} \bigr\}.
\]
We claim that $B(s)/s$ is a nondecreasing mapping from $(0,1)$ onto 
some neighbourhood of zero, and hence the sets $G_t$ are nonempty
for every $t\in(0,1)$.
Indeed, if $\lim_{s\to 0+} B(s)/s > 0$, then
$B(s)\ge cs$ on $(0,1)$ for some $c>0$ and thus
\[
	\int_{0} \frac{B(s)}{s^{\ngnd+1}}\,\d s
		\ge c \int_{0} \frac{\d s}{s^\ngnd} = \infty
\]
which contradicts \eqref{eq:ineqDBzero}.

Let us define $\tau=\tau_t=\sup G_t$.
Observe that, by the continuity of Young functions,
\begin{equation} \label{feb10}
	\frac{B(\tau)}{\tau}
		= \frac{D(t)}{t},
		\quad 0<t<1.
\end{equation}
Also,
\begin{equation}
		\label{feb111}
	\limsup_{t\to\infty}	\frac{B(\tau_t)}{\tau_t}\cdot \frac{t}{B(Kt)}=\infty
\end{equation}
for every $K\ge 1$.
Indeed, suppose that there is some $K\ge 1$ for which \eqref{feb111} is violated.
We then have some $L>0$ such that
\[
	\frac{B(Kt)}{t} \ge L \frac{B(\tau)}{\tau},
		\quad 0<t<1,
\]
which in connection with \eqref{feb10} gives $B(Ks)\ge L\, D(s)$ on $(0,1)$. 
Thus
\begin{multline*}
	\frac{D(CKt)}{t^\ngnd}
		\ge \int_{0}^{Kt} \frac{B(s)}{s^{\ngnd+1}}\,\d s 
		= K^\nng \int_{0}^{t} \frac{B(Ks)}{s^{\ngnd+1}}\,\d s
		\\
		\ge LK^\nng \int_{0}^{t} \frac{D(s)}{s^{\ngnd+1}}\,\d s,
		\quad 0<t<1,
\end{multline*}
which contradicts \eqref{eq:supintD} and therefore \eqref{feb111} holds true.

Next, by\eqref{feb111}, we can take a decreasing sequence
$t_k\in(0,1)$, $k\in\N$, such that
\begin{equation} \label{feb112}
	\lim_{k\to\infty} \frac{B(\tau_k)}{\tau_k} \cdot \frac{t_k}{B(kt_k)}=\infty,
\end{equation}
where we set $\tau_k=\tau_{t_k}$.
Without loss of generality we may assume that $2t_k<\tau_k$ for every $k\in\N$.
For contradiction, suppose that there exists a subsequence $\{k_j\}$ such that
$\tau_{k_j}\le 2t_{k_j}$. Then, since $B(s)/s$ does not increase and thanks to
\eqref{kt},
\[
	\frac{B(\tau_{k_j})}{\tau_{k_j}}\cdot \frac{t_{k_j}}{B(k_j t_{k_j})}
		\le \frac{B(2t_{k_j})}{2t_{k_j}}\cdot \frac{t_{k_j}}{B(k_j t_{k_j})}
		\le \frac{B(2t_{k_j})}{B(2t_{k_j})}\cdot \frac{1}{k_j}
		= \frac{1}{k_j}\to 0
			\quad\text{as}\quad k\to\infty,
\]
which is impossible due to \eqref{feb112}.
We may also require that $t_{k+1}$ is chosen in a way that
$2t_{k+1} \le \tau_{k+1} < t_{k}$, which is ensured if
$\tau_t \to 0$ as $t\to 0$. To observe that, by \eqref{feb10}, we need
to have $\lim_{t\to 0+} D(t)/t=0$ which is however guaranteed
by the stronger condition \eqref{eq:limD0}.
Furthermore, from \eqref{eq:limD0} take $t_{k+1}$ small enough so that
\begin{equation} \label{feb14}
	\frac{D(t_{k+1})}{t_{k+1}^\ngnd}
		\le \frac{1}{2}\cdot \frac{D(t_k)}{t_k^\ngnd}.
\end{equation}

We now define a function $B_1$ by the formula
\[
	B_1(t) =
	\begin{cases}\displaystyle
		B(t_k) +
			\frac{B(\tau_k) - B(t_k)}{\tau_k - t_k}\,(t - t_k),
			& t\in(t_k,\tau_k),\;k\in\N,\\
		B(t),
			& \text{otherwise}.
	\end{cases}
\]
Obviously, $B_1$ is a well-defined Young function and $B_1\ge B$.
Moreover, for $k\in\N$, $2B(t_k) \le B(\tau_k)$ by \eqref{kt} and 
therefore
\begin{equation*}
	\frac{B_1(2t_k)}{B(kt_k)}
		= \frac{B(t_k) +
				\frac{B(\tau_k) - B(t_k)}{\tau_k - t_k}\, t_k}
				{B(kt_k)}
		\ge \frac{ B(\tau_k) - B(t_k)}{B(kt_k)}
			\cdot \frac{t_k}{\tau_k}
		\ge \frac12\cdot \frac{B(\tau_k)}{\tau_k}
			\cdot \frac{t_k}{B(kt_k)}
\end{equation*}
and the latter tends to infinity as $k\to\infty$ by \eqref{feb112}.
Consequently
\[
	\limsup_{t\to 0+} \frac{B_1(t)}{B(\lambda t)} = \infty
\]
for every $\lambda\ge 1$ and $B_1$ essentially dominates $B$.
It remains to show that $B_1$ fulfills \eqref{eq:ineqDB1zero}.
Let $t\in(0,1)$ be fixed and let $j\in\N$ be such that $t\in[t_j,t_{j+1})$.
Then we have
\begin{equation} \label{feb11}
	\int_{0}^{t} \frac{B_1(s)}{s^{\ngnd+1}}\,\d s
		\le \int_{0}^{t} \frac{B(s)}{s^{\ngnd+1}}\,\d s
			+ \sum_{k=j}^\infty \frac{B(\tau_k) - B(t_k)}{\tau_k - t_k}
				\int_{t_k}^{\tau_k} \frac{s-t_k}{s^{\ngnd+1}}\,\d s.
\end{equation}
By the assumption, the former integral is dominated by
the right hand side of \eqref{eq:ineqDBzero}.
Let us follow with estimates of the latter sum.
Thanks to $2t_k<\tau_k$ and \eqref{feb10},
\begin{equation} \label{feb12}
	\frac{B(\tau_k) - B(t_k)}{\tau_k - t_k}
		\le \frac{2B(\tau_k)}{\tau_k}
		= \frac{2D(t_k)}{t_k}
\end{equation}
and since $\ngnd>1$, we have
\begin{equation} \label{feb13}
	\int_{t_k}^{\tau_k}\frac{(s-t_k)}{ s^{\ngnd+1}}\,\d s
		\le \int_{t_k}^{\infty} \frac{1}{ s^\ngnd}\,\d s
		\le \frac{1}{t_k^{\ngnd-1}}.
\end{equation}
Combination of \eqref{feb11} and \eqref{feb12} with \eqref{feb13} then gives
\begin{equation} \label{feb15}
	\int_{0}^{t} \frac{B_1(s)}{s^{\ngnd+1}}\,\d s
		\le \frac{D(Ct)}{t^\ngnd} + 2\sum_{k=j}^\infty \frac{D(t_k)}{t_k^\ngnd}.
\end{equation}
It follows from \eqref{feb14} that
\[
	\frac{D(t_k)}{t_k^\ngnd}
		\le 2^{k-j} \frac{D(t_j)}{t_j^\ngnd}
		\le 2^{k-j} \frac{D(t)}{t^\ngnd},
		\quad j\le k <\infty,
\]
whence
\begin{equation} \label{feb16}
	\sum_{k=j}^\infty \frac{D(t_k)}{t_k^\ngnd}
		\le \frac{D(t)}{t^\ngnd} \sum_{j=k}^\infty 2^{k-j}
		\le \frac{2D(t)}{t^\ngnd}.
\end{equation}
The inequality \eqref{eq:ineqDB1zero} thus follows by \eqref{feb15} and \eqref{feb16}.
\end{proof}

\begin{lemma} \label{lemm:construction_bounded}
Let $B$ be a Young function satisfying
\begin{equation} \label{feb17}
	\int_{0}^{1} \frac{B(s)}{s^{\ngnd+1}}\,\d s < \infty.
\end{equation}
Then there is a Young function $B_1$ such that $B_1$ essentially dominates
$B$ and also
\begin{equation} \label{feb18}
	\int_{0}^{1} \frac{B_1(s)}{s^{\ngnd+1}}\,\d s < \infty.
\end{equation}
\end{lemma}

\begin{proof}
Assume that $B$ is given and fulfills \eqref{feb17}.
Let us define $d_k = 1/\log(k+1)$, $k\in\N$, and
let $t_k$, $k\in\N$, satisfy $t_k \le t_{k-1}/k$ and
\[
	\int_{0}^{t_k} \frac{B(s)}{s^{\ngnd+1}}\,\d s \le d_k
		\quad\text{for $k\in\N$}.
\]
Then define $D_1$ by
\[
	D_1(t) = \sum_{k=1}^\infty d_k t^\ngn \chi_{[t_{k+1},t_k)}(t)
		+ Ct^\ngn\chi_{[1,\infty)}(t),
		\quad t\ge 0,
\]
where $C$ is a constant which dominates the integral in \eqref{feb17},
and set $D$ by
\[
	D(t) = \int_{0}^{2t} \frac{D_1(s)}{s}\,\d s,
		\quad t\ge 0.
\]
Since $D_1$ is nondecreasing, $D$ is a Young function and $D_1(t) \le D(t)$.

We shall show that $B$ and $D$ satisfy the assumptions of Lemma~\ref{lemm:construction}.
Clearly $D(t)/t^\ngnd$ is nondecreasing and
$\lim_{t\to 0+} D(t)/t^\ngnd = \lim_{k\to\infty} d_k = 0$.
Next, let $t\in(0,1)$ be arbitrary and let $j\in\N$ be such that
$t\in[t_{j+1},t_j)$. Then
\begin{align*}
	\int_{0}^{t} \frac{D(s)}{s^{\ngnd+1}}\,\d s
		& \ge \int_{0}^{t} \frac{D_1(s)}{s^{\ngnd+1}}\,\d s
		\ge \sum_{k=j+1}^\infty \int_{t_{k+1}}^{t_k} \frac{D_1(s)}{s^{\ngnd+1}}\,\d s
			\\
		& \ge \sum_{k=j+1}^\infty d_k \int_{t_{k+1}}^{t_k} \frac{\d s}{s}
		\ge \sum_{k=j+1}^\infty d_k \log(k+1) = \infty
\end{align*}
which gives \eqref{eq:supintD} and also
\[
	\int_{0}^{t} \frac{B(s)}{s^{\ngnd+1}} \,\d s
		\le \int_{0}^{t_j} \frac{B(s)}{s^{\ngnd+1}} \,\d s
		= d_j
		= \frac{D_1(t)}{t^\ngnd}
		\le \frac{D(t)}{t^\ngnd}
\]
which is \eqref{eq:ineqDBzero}. Lemma~\ref{lemm:construction} gives us
a Young function $B_1$ essentially dominating $B$ such that
\[
	\int_{0}^{t} \frac{B_1(s)}{s^{\ngnd+1}} \,\d s
		\le \frac{D(5t)}{t^\ngnd},
		\quad 0<t<1.
\]
Then we have \eqref{feb18} as a special case.
\end{proof}

\begin{proof}[Proof of Theorem~\ref{thm:OOT}]
Let $A$ be a Young function satisfying \eqref{eq:Acond}
and assume \eqref{eq:Angindex}.
By \citep[Proposition 4.1]{Cia:18}, \eqref{eq:Angindex} is
equivalent to the inequality
\begin{equation} \label{feb5}
	\int_{0}^{t} \frac{A_\ng(s)}{s^{\ngnd+1}}\,\d s
		\le \frac{A_\ng(Ct)}{t^\ngnd},
		\quad t>0,
\end{equation}
for some constant $C>0$. Theorem~\ref{thm:Orlicz_reduction} then
guarantees \eqref{eq:MgAAng}.

Let us prove the optimality. Suppose that $L^B(\Rn)$ satisfies
\eqref{eq:MgAB}. Then, by Theorem~\ref{thm:Orlicz_reduction},
there is a constant $C>0$ such that
\begin{equation} \label{eq:ineqAB}
	\int_{0}^{t} \frac{B(s)}{s^{\ngnd+1}}\,\d s
		\le \frac{A_\ng(Ct)}{t^\ngnd},
		\quad t>0.
\end{equation}
Thus,
\begin{multline*}
	A_\ng(2Ct)
		\ge t^\ngnd \int_{0}^{2t} \frac{B(s)}{s^{\ngnd+1}}\,\d s
		\ge t^\ngnd \int_{t}^{2t} \frac{B(s)}{s^{\ngnd+1}}\,\d s
		\\
		\ge B(t)\,t^\ngnd \int_{t}^{2t} \frac{\d s}{s^{\ngnd+1}}
		\ge B(t) C_{n,\gamma}
\end{multline*}
for all $t>0$
and hence $A_\ng$ dominates $B$ which gives $L^{A_\ng}(\Rn) \to L^B(\Rn)$.

Suppose that $A$ satisfies \eqref{eq:Acond} and \eqref{eq:Angindex} fails.
Then, by \citep[Proposition 4.1]{Cia:18}, the inequality
\eqref{feb5} is violated for every $C>0$.
Note that in this case is $L^{A_\ng}(\Rn) \ne L^\infty(\Rn)$
and hence $A_\ng$ is finite-valued.
The failure of \eqref{feb5} occurs under one of these two conditions,
namely
\begin{equation} \label{feb6}
	\sup_{1<t<\infty} \frac{t^\ngnd}{A_\ng(Ct)} \int_{1}^{t} \frac{A_\ng(s)}{s^{\ngnd+1}}\,\d s
		= \infty
		\quad\text{for every $C>0$}
\end{equation}
or
\begin{equation} \label{feb7}
	\sup_{0<t<1} \frac{t^\ngnd}{A_\ng(Ct)} \int_{0}^{t} \frac{A_\ng(s)}{s^{\ngnd+1}}\,\d s
		= \infty
		\quad\text{for every $C>0$}.
\end{equation}
Assume that $B$ is a Young function such that \eqref{eq:MgAB} holds,
i.e., by Theorem~\ref{thm:Orlicz_reduction}, the inequality \eqref{eq:ineqAB} holds
for some $C>0$.
In both cases, we will show that there is a Young function $B_1$ such that
$L^{B_1}(\Rn) \subsetneq L^B(\Rn)$ and also
\begin{equation} \label{eq:MgAB1}
	M_\gamma\colon L^A(\Rn) \to L^{B_1}(\Rn).
\end{equation}
or equivalently, by Theorem~\ref{thm:Orlicz_reduction},
\begin{equation} \label{eq:ineqAB1}
	\int_{0}^{t} \frac{B_1(s)}{s^{\ngnd+1}}\,\d s
		\le \frac{A_\ng(Ct)}{t^\ngnd},
		\quad t>0,
\end{equation}
where $C$ is a possibly different constant.
If \eqref{feb6} holds, we modify the Young function $B$ only near infinity
and, similarly, if \eqref{feb7} holds, we do that near zero.

Let us start with the ``near infinity'' case.
First observe that $A_\ng(t)/t^\ngnd$ is nondecreasing,
by the same argument used in Remark~\ref{rem:AngYoung}. Thus
$A_\ng(t)/t^\ngnd$ is either bounded near infinity or
\begin{equation} \label{feb8}
	\lim_{t\to\infty} \frac{A_\ng(t)}{t^\ngnd} = \infty.
\end{equation}
In the former case, \eqref{eq:ineqAB} reads as
\begin{equation} \label{eq:ineqABzero}
	\int_{0}^{t} \frac{B(s)}{s^{\ngnd+1}}\,\d s
		\le \frac{A_\ng(Ct)}{t^\ngnd},
		\quad 0<t<1,
\end{equation}
and
\[
	\int_{1}^{\infty} \frac{B(s)}{s^{\ngnd+1}}\,\d s < \infty.
\]
By \citep[Theorem~6.4 (ii)]{Cia:98a},
there is a modification of $B$ on $(1,\infty)$, say $B_1$, such that $B_1$ essentially
dominates $B$ and also
\[
	\int_{1}^{\infty} \frac{B_1(s)}{s^{\ngnd+1}}\,\d s < \infty.
\]
Let us moreover set $B_1 = B$ on $(0,1)$.
Then \eqref{eq:ineqABzero} holds with $B$ replaced by $B_1$ and
hence \eqref{eq:ineqAB1} and therefore \eqref{eq:MgAB1}.

Now, assume \eqref{feb6} and \eqref{feb8}. The condition \eqref{eq:ineqAB} now splits
into \eqref{eq:ineqABzero} and
\begin{equation} \label{eq:ineqABinfty}
	\int_{1}^{t} \frac{B(s)}{s^{\ngnd+1}}\,\d s
		\le \frac{A_\ng(Ct)}{t^\ngnd},
		\quad 1<t<\infty.
\end{equation}
By \citep[Theorem~4.1]{Mus:16}, there exists a modified Young function $B_1$ such that
$B_1$ is essentially larger than $B$ and also satisfies \eqref{eq:ineqABinfty} with
$B_1$ in place of $B$. If we keep $B_1=B$ on $(0,1)$, then \eqref{eq:ineqABzero} remains valid
for $B_1$. That gives us \eqref{eq:ineqAB1} also in this case.

Let us now work ``near zero''. We again distinguish two cases, when $A_\ng(t)/t^\ngnd$
is equivalent to a constant function on $(0,1)$ and when
\begin{equation} \label{feb9}
	\lim_{t\to 0} \frac{A_\ng(t)}{t^\ngnd} = 0.
\end{equation}
In the constant case, \eqref{eq:ineqAB} boils down to \eqref{eq:ineqABinfty}
and
\[
	\int_{0}^{1} \frac{B(s)}{s^{\ngnd+1}}\,\d s < \infty.
\]
Lemma~\ref{lemm:construction_bounded} ensures
that there is an essentially larger Young function than $B$, $B_1$ say, such that
\[
	\int_{0}^{1} \frac{B_1(s)}{s^{\ngnd+1}}\,\d s < \infty.
\]
Let $B_1=B$ on $(1,\infty)$ and thus \eqref{eq:ineqABinfty} holds for $B$ replaced by $B_1$.
Therefore \eqref{eq:ineqAB1} also in this case.

Finally assume \eqref{feb7} and \eqref{feb9}. The inequality \eqref{eq:ineqAB} splits into
\eqref{eq:ineqABzero} and \eqref{eq:ineqABinfty}. On using Lemma~\ref{lemm:construction},
one gets a modified Young function $B_1$ which essentially enlarges $B$ on $(0,1)$
and still satisfies \eqref{eq:ineqABzero}. Again, if we set $B_1=B$ on $(0,\infty)$,
we obtain \eqref{eq:ineqAB1}.

Let us prove the ``in particular'' statement. The condition $I_A<\ng$ holds
if and only if $i_{\tilde{A}} > \ngn$ is satisfied.
Then, by \citep[Lemma~4.2 (ii)]{Cia:18},
\[
	\inf_{0<s\le t} \tilde{A}^{-1}(s)\, s^{\gn-1}
		\simeq \tilde{A}^{-1}(t)\, t^{\gn-1}
		\quad\text{for $t>0$}
\]
which, thanks to \eqref{eq:YoungCompl}, rewrites as
\[
	\sup_{0<s\le t} A^{-1}(s)\,s^{-\gn}
		\simeq A^{-1}(t)\,t^{-\ng}
		\quad\text{for $t>0$}.
\]
Thus, \eqref{eq:supout} follows by the definition
of $\Ang$ and by Remark~\ref{rem:AngYoung}.
Also, if $i_{\tilde{A}} > \ngn$, then
$i_\Ang>\ng$ is equivalent to $i_A>1$
due to \citep[Lemma~4.3]{Cia:18} and the duality.

The necessity of the condition \eqref{eq:Acond}
for the existence of any Orlicz space $L^B(\Rn)$
satisfying \eqref{eq:MgAB} follows by
Theorem~\ref{thm:Orlicz_reduction}.
\end{proof}

\begin{proof}[Proof of Theorem~\ref{thm:ORR}]
The boundedness of $M_\gamma$ in \eqref{eq:MgBngBopt}
and the optimality of the domain space is a consequence
of Theorem~\ref{thm:OOD}.

As for the range, set $A=B_\ng$ and let us compute $A_\ng$ by
\eqref{eq:AngDef}. We have
\[
	A_\ng^{-1}(t)
		\simeq G_\ng(t)
		= \sup_{0<s\le t} B_\ng^{-1}(s)\,s^{-\gn}
		\simeq \sup_{0<s\le t} F_\ng^{-1}(s)
		= F_\ng^{-1}(t),
		\quad t>0,
\]
whence $A_\ng$ is equivalent to $F_\ng$.
Therefore $i_{A_\ng} > \ngn$ if and only if
$i_{F_\ng} > \ngnd$ which is due to Lemma~\ref{lemm:inheritedind}
the same as $i_B > \ngn$.
The optimality of the range in \eqref{eq:MgBngBopt} is then
driven by the condition \eqref{eq:Bindex} thanks to Theorem~\ref{thm:OOT}.
\end{proof}

\begin{proof}[Proof of Theorem~\ref{thm:ODR}]
Let $A$ be a Young function such that \eqref{eq:Acond} holds true.
Then $\Asup$ satisfies \eqref{eq:Acond} as well
and the Young function $(\Asup)_\ng$ associated to $\Asup$
as in \eqref{eq:AngDef} satisfies
\begin{align*}
	(\Asup)_\ng^{-1}(t)
		& \simeq \sup_{0<s\le t} s^{-\ng}\Asupi(s) 
			\\
		& \simeq \sup_{0<s\le t} \sup_{0<y\le s} y^{-\ng}A^{-1}(y) 
			\\
		& = \sup_{0<y\le t} y^{-\ng}A^{-1}(y) 
		\simeq A_\ng^{-1}(t),
			\quad t>0,
\end{align*}
in other words, $(\Asup)_\ng$ is equivalent to $A_\ng$.
Therefore, by Theorem~\ref{thm:Orlicz_reduction}, criterion (iii),
\eqref{eq:MgAB} holds if and only if \eqref{eq:MgAsupB}.
Also,
\[
	(\Asup)_\ng^{-1}(t)
		\simeq t^{-\ng} \Asupi(t),
			\quad t>0,
\]
and thus
\[
	h_{(\Asup)_\ng}(t)
		= \sup_{s>0} \frac{(\Asup)_\ng^{-1}(st)}{(\Asup)_\ng^{-1}(s)}
		\simeq t^{-\ng} \sup_{s>0} \frac{\Asupi(st)}{\Asupi(s)}
		= t^{-\ng} h_{\Asup}(t),
		\quad t>0.
\]
Whence, by the definition of the lower Boyd index \eqref{BIdef},
\[
	\frac{1}{i_{(\Asup)_\ng}} + \gn = \frac{1}{i_{\Asup}}
\]
and therefore $i_{(\Asup)_\ng} > \ngn$ if and only if $i_\Asup > 1$.
The claim now follows by Theorem~\ref{thm:OOT}.
\end{proof}


\end{document}